\documentclass[11pt]{amsart}
\usepackage{amscd,amssymb,latexsym,verbatim}
\usepackage{hyperref}
\usepackage[active]{srcltx}

\theoremstyle{plain}
\newtheorem{thm}{Theorem}[section]

\newtheorem{prop}[thm]{Proposition}

\theoremstyle{definition}
\newtheorem{definition}[thm]{Definition} 
\newtheorem{example}[thm]{Example}
\newtheorem{question}[thm]{Question}

\theoremstyle{remark}
\newtheorem{rem}[thm]{Remark}

\newcommand\rightmap[1]{\smash{\mathop{\rightarrow}\limits^{#1}}}

\newcommand{\CC}{\mathcal{C}}
\newcommand{\DD}{\mathcal{D}}
\newcommand{\EE}{\mathcal{E}}

\newcommand{\LL}{\mathcal{L}}

\newcommand{\TT}{\mathbb{T}}
\newcommand{\bb}{\mathbb{B}}
\newcommand{\bc}{\mathbb{C}}

\newcommand{\bn}{\mathbb{N}}

\newcommand{\bz}{\mathbb{Z}}

\newcommand{\bp}{\mathbb{P}}

\DeclareMathOperator{\lcm}{lcm}
\DeclareMathOperator{\rank}{rank}

\DeclareMathOperator{\Hom}{Hom}

\DeclareMathOperator{\GL}{GL}
\DeclareMathOperator{\ab}{ab}

\DeclareMathOperator{\Char}{\mathcal V}

\DeclareMathOperator{\orb}{orb}
\DeclareMathOperator{\NC}{NC}
\DeclareMathOperator{\LF}{LF}
\DeclareMathOperator{\nul}{null}
\DeclareMathOperator{\mult}{mult}
\DeclareMathOperator{\Sing}{Sing}
\DeclareMathOperator{\Reg}{Reg}
\DeclareMathOperator{\depth}{depth}

\numberwithin{equation}{section}

\newcommand\enet[1]{\renewcommand\theenumi{#1}
\renewcommand\labelenumi{\theenumi}}

\title{Orbifold groups, quasi-projectivity and covers}

\author[E. Artal]{Enrique Artal Bartolo}

\author[J.I. Cogolludo]{Jos\'e I. Cogolludo-Agust\'in}
\address{Departamento de Matem\'aticas\\
Universidad de Zaragoza\\
Campus Plaza San Francisco s/n\\
E-50009 Zaragoza SPAIN}
\email{artal@unizar.es,jicogo@unizar.es}

\author[D. Matei]{Daniel Matei}
\address{Institute of Mathematics of the Romanian Academy,
P.O. Box 1-764,  RO-014700, Bucharest, Romania} 
\email{Daniel.Matei@imar.ro}
\urladdr{http://www.imar.ro/\~{}dmatei}

\thanks{
Partially supported by
MTM2010-21740-C02-02. The third author is also partially supported
by grant CNCSIS PNII-IDEI 1188/2008 and FMI 53/10 (Gobierno de Arag{\'o}n).}

\begin{document}

\begin{abstract}
We discuss properties of complex algebraic orbifold groups, their characteristic varieties, and their 
abelian covers. In particular, we deal with the question of (quasi)-projectivity of orbifold groups.
We also prove a structure theorem for the variety of characters of normal-crossing quasi-projective 
orbifold groups. Finally, we extend Sakuma's formula for the first Betti number of abelian covers
of orbifold fundamental groups. Several examples are presented, including a compact orbifold
group which is not projective and a Zariski pair of plane curves in $\bp^2$ that can be told by 
considering an unbranched cover of $\bp^2$ with an orbifold structure.
\end{abstract}

\maketitle

\section*{Introduction}

Any finitely presented group $G$ is the fundamental group of a closed oriented~$4$-manifold. If we ask these
manifolds to have extra-properties, some restrictions may apply. For example, such a group is said to be 
\emph{K\"ahler} if it is the fundamental group of a compact K\"ahler manifold. 
A subclass of this family is given by the \emph{projective} groups, i.e., the fundamental group of 
a complex projective smooth surface (or equivalently, of a projective $n$-manifold). 

The family of projective groups is a subfamily of \emph{quasi-projective} groups. Recall that a quasi-projective
manifold is the difference of two projective varieties. The study of K\"ahler, projective and quasi-projective groups
is closely related to orbifold groups, or more precisely to orbicurve groups, 
i.e. orbifold fundamental groups of complex $1$-dimensional orbifolds. Recently, orbifold groups
(in any complex dimension) have been considered (cf.~\cite{Uludag-orbifolds,Campana-special}
also~\cite{Kato-uniformization} for real orbifolds).

The first purpose of this paper is to define and study the properties of the different classes of complex orbifold 
fundamental groups such as compact, locally finite, and normal crossing. In particular, we prove that orbifold fundamental
groups are quasi-projective, but compact orbifold groups in general are not projective (see~\S\ref{sec-groups}).
In this context, we develop in \S\ref{sec-saturation} the concept of \emph{saturated} orbifolds, which will allow 
one to transform orbifolds without altering their fundamental group.

Our second purpose (see~\S\ref{sec-charvar}) is to extend two classical results regarding the variety of characters 
on smooth quasi-projective fundamental groups (due to Arapura~\cite{ara:97} and the authors~\cite{ACM-prep}) and
normal-crossing compact K\"ahler orbiface groups (due to Campana~\cite{Campana-special}) to the general case
of normal-crossing quasi-projective orbifold groups. 

Finally in \S\ref{sec-sakuma}, we extend Sakuma's formula (cf.~\cite{Sakuma-homology,Eriko-alexander}) to orbifold 
fundamental groups and their abelian covers in terms of their orbifold characteristic varieties. In order to do 
so, in \S\ref{sec-covers} we present the concepts of unbranched and branched coverings as well as the possible
uniformizations (Galois, regular, and virtually regular). Such formulas are illustrated with examples in dimensions
one and two.

\section{Orbifold groups}\label{sec-groups}

\begin{definition}\label{def-orbicurve}
Let $\bar{X}$ be a projective Riemann surface and let $\varphi:\bar{X}\to\bz_{\geq 0}$ 
be a function such that $S_\varphi:=\{p\in X\mid \varphi(p)\neq 1\}$ is finite. The pair $(\bar{X},\varphi)$ 
is said to be a $1$-dimensional \emph{orbifold} or an \emph{orbicurve}. The \emph{positive part} of the orbicurve
is $X_{\varphi}^+:=\bar{X}\setminus\varphi^{-1}(0)$ and we say that the orbifold is \emph{compact}
if $X_{\varphi}^+=\bar{X}$. The set $S_{\varphi}^{>1}:=X_{\varphi}^+\cap S_\varphi$ is called the \emph{singular part}
and $\varphi(p)$ is the \emph{orbifold index} of $p\in X_{\varphi}^+$.  
\end{definition}

The geometrical interpretation is the following. The source of 
the charts centered at $p\in X_{\varphi}^+$ are of the type~$\Delta/\mu_{\varphi(p)}$ where $\mu_n:=\{z\in\bc\mid z^n=1\}$, 
$\Delta$ is an open disk centered at~$0$ and $\mu_{\varphi(p)}$ acts on $\Delta$ by multiplication. 
This interpretation suggests the following definition. 

\begin{definition} Let $(\bar{X},\varphi)$ be an orbicurve. Let $X_\varphi:=\bar{X}\setminus S_\varphi$ and
$G:=\pi_1(X_\varphi; p_0)$ for some $p_0\in X_\varphi$. For each $p\in S_\varphi$ choose a meridian
$x_p\in G$ (its conjugacy class is well defined). Then, the \emph{orbifold fundamental group}
of $(X,\varphi)$ is defined as
$$
\pi_1^{\orb}(\bar{X},\varphi;p_0):=G/\langle x_p^{\varphi(p)}\rangle.
$$
A group is said to be an \emph{orbicurve group} if it is the orbifold fundamental group of an orbicurve.
\end{definition}

\begin{rem}
If the group can be described as the orbifold group of a compact orbicurve
then we will refer to it as a \emph{compact orbicurve group}.
\end{rem}

\begin{prop}\label{prop-orbicurve}
Any orbicurve group is quasi-projective.
\end{prop}

In order to prove this result we introduce the following concept.

\begin{definition}
Let $X$ be a smooth quasi-projective surface, let $\DD$ be a divisor in $X$ and let $D\subset X$ an irreducible 
component of $\DD$. 
An \emph{$n$-fold blow-up} $\rho$ of $(X,\DD)$ on $D$ is a composition of blowing-ups 
$\rho_j:X_j\to X_{j-1}$, $1\leq j\leq n$, such that $X_0:=X$, the center of $\rho_1$ is a smooth point of $\DD$ 
in $D$ and if $E_j$ is the exceptional component of $\rho_j$ then, for $j>1$ the center of $\rho_j$ is the 
intersection of $E_j$ with the strict transform of $D$. The component $E_j$ is called the 
\emph{$j$-th exceptional component} of~$\rho$.
\end{definition}

\begin{rem}
\label{rem-nfold}
The most useful property of an $n$-fold blow-up $\rho$ on an irreducible divisor $D$ is that, if $\mu$ is a 
meridian around $D$, then $\mu^j$ is a meridian around the $j$-th exceptional component of~$\rho$.
\end{rem}

\begin{proof}[Proof of Proposition{\rm~\ref{prop-orbicurve}}]
Let $(\bar{X},\varphi)$ be an orbicurve. Let $Y:=X_\varphi^+\times\bp^1$ be a surface, 
and let $\DD:=S_{\varphi}^{>1} \times\bp^1\subset Y$. For each $p\in S_{\varphi}^{>1}$ consider 
a $\varphi(p)$-fold blow-up $\rho_p:\tilde{Y}\to Y$ on the divisor $F_p:=\{p\}\times\bp^1$.
Let $E_j^p$ be the $j$-th exceptional component of $\rho_p$. 
Let $x_p$~be a meridian of $\{p\}\times\bp^1$ in $\pi_1(\tilde{Y})$. Following the previous remark,
$x_p^j$~is a meridian of $E_j^p$ in $\pi_1(\tilde{Y})$. Set
$Z:=\tilde{Y}\setminus\bigcup_{p\in S^{>1}_\varphi}\left(F_p\cup\bigcup_{j=1}^{\varphi(p)-1}E_j^p \right)$.
The kernel of the epimorphism $\pi_1(X_\varphi^+)\cong\pi_1(\tilde{Y})\twoheadrightarrow\pi_1(Z)$
is normally generated by the meridians $x_p^{\varphi(p)}$ of~$E_{\varphi(p)}^p$.
Then $Z$ is a smooth quasi-projective surface and $\pi_1(Z)$ is isomorphic to~$\pi_1^{\orb}(\bar{X},\varphi)$.
\end{proof}

\begin{rem}\label{rem-fm}
As shown in~\cite[Theorem~II.2.3]{MF}, compact orbicurve groups are projective groups.
\end{rem}

We will define orbifolds and orbifold groups following Campana
(cf.~\cite{Campana-special} and bibliography therein).
Since we are mostly interested in quasi-projective groups, after using Zariski-Lefschetz theory we can
restrict our attention to the curve and surface case. However, since we will deal with orbifold covers
(see~\S\ref{sec-covers}) orbifolds with abelian quotient singularities will also be allowed.

\begin{definition}
Let $\bar{X}$ be a projective variety (a normal variety with only abelian quotient singularities)
and let $\DD=\bigcup_{j=1}^r D_i$ be the decomposition of a hypersurface in irreducible components. 
Let us consider a function
$\varphi:\{D_1,\dots,D_r\}\to \bz_{\geq 0}$, $n_i:=\varphi(D_i)$. 
An \emph{orbifold} is simply a pair $(\bar{X},\varphi)$. The \emph{positive part} of the orbifold
is defined as $X_\varphi^+:=\bar{X}\setminus\varphi^{-1}(0)$. 
The orbifold is said to be \emph{compact} if $\bar{X}=X_\varphi^+$. 
The orbifold will be a \emph{normal-crossing orbifold} ($\NC$ for short) if $\DD$ 
is a normal crossing divisor with smooth components.
\end{definition}

\begin{rem}
Note that, for technical reasons, the components of $\DD$ are allowed to have index one 
(that is, $n_j=1$). However, this plays no important role in the definition of an orbifold.
Hence, if no ambiguity seems likely to arise, we denote by the same symbols an orbifold and
its analogous where $\varphi^{-1}(1)$ is disregarded.
What is really important in the definition is the quasi-projective variety $X_\varphi^+$ and 
the components $D_j$ with $n_j>1$. Following the definitions for the orbicurve case we also define
$$
S_\varphi:=\{D_j\mid n_j\neq 1\},\ 
S_\varphi^{>1}:=\{D_j\mid n_j>1\},\ 
X_\varphi:=\bar{X}\setminus\left(\bigcup S_\varphi\right),\ 
\mathring{X}_\varphi:=\bar{X}\setminus \DD.
$$
Note that $\mathring{X}_\varphi\subset X_\varphi\subset X_\varphi^+$. 
In $\pi_1(X_\varphi)$ and $\pi_1(\mathring{X}_\varphi)$ one has special conjugacy classes: for each $D_i$ we 
consider the meridians of $D_i$ in either $\pi_1(X_\varphi)$ or $\pi_1(\mathring{X}_\varphi)$. Note that the 
kernel of the epimorphism $\pi_1(X_\varphi)\twoheadrightarrow\pi_1(\bar{X})$ is the subgroup generated by the meridians 
of $D_j$, $n_j\neq 1$ whereas the kernel of the epimorphism 
$\pi_1(\mathring{X}_\varphi)\twoheadrightarrow\pi_1(\bar{X})$ is the subgroup generated by the meridians 
of~$D_1,\dots,D_r$. 
\end{rem}

\begin{definition}
Under the notation above, given an orbifold $(\bar{X},\varphi)$ we define its \emph{orbifold fundamental group}
as the group $\pi_1^{\orb}(\bar{X},\varphi;p_0)$, 
$p_0\in \Reg(\mathring{X}_\varphi):=\mathring{X}_\varphi\setminus \Sing(\mathring{X}_\varphi)$ 
obtained as the quotient of $\pi_1(\mathring{X}_\varphi;p_0)$ by the subgroup normally generated by 
$\{\mu_j^{n_j}\}_{1\leq j\leq r},$ where $\mu_j$ is a meridian of~$D_j$. Note that $\pi_1(\mathring{X}_\varphi)$ 
can also be replaced by $\pi_1(X_\varphi)$ in this definition.

For $p\in X_\varphi^+$ one can define the \emph{local orbifold fundamental group} $\pi_1^{\orb}(\bar{X},\varphi)_p$ 
as the quotient of $\pi_1(\Reg(\mathring{X}_\varphi))_p$ by the subgroup 
normally generated by the appropriate powers $\mu_j^{n_j}$ of the meridians $\mu_j$ of $\DD$ in a small ball around~$p$. 
The orbifold $(\bar{X},\varphi)$ shall be called \emph{locally finite at $p$} if $\pi_1^{\orb}(\bar{X},\varphi)_p$ is 
a finite group, and \emph{locally finite} (or simply $\LF$) if it is locally finite at $p$, $\forall p\in X_\varphi^+$.
\end{definition}

We need to extend the notion of the orbifold
index of a point in an orbifold as we did for orbicurves in Definition~\ref{def-orbicurve}.

\begin{definition}
\label{def-nu}
Let $(\bar{X},\varphi)$ be an $\NC$-orbifold and let $p\in\bar{X}$. We define the \emph{orbifold index} 
$\nu(p)=\nu_{(\bar{X},\varphi)}(p)$ of $p$ as follows:
$$
\nu(p):=
\begin{cases}
\iota_p &\text{ if } p\in\bar{X}\setminus\DD,\\
n_j \cdot \iota_p &\text{ if } p\in D_j\setminus\bigcup_{i\neq j}D_i,\\
n_i\cdot n_j \cdot \iota_p &\text{ if } p\in D_i\cap D_j, i\neq j,
\end{cases}
$$
where 
$\iota_p=|A|$ if $(\bar{X},p)\cong(\bc^2/A,0)$, the quotient by the linear
action of  a small abelian subgroup~$A\subset\GL(2;\bc)$
(note that $\iota_p=1$ iff $p\in \Reg(\bar{X})$).
\end{definition}

\begin{rem}
If $p\in \Reg(X_\varphi)$ (or $\Reg(\mathring{X}_\varphi)$) then  $\pi_1^{\orb}(X,\varphi)_p$ is a trivial group. 
\end{rem}

\begin{prop}
If $p\in X_\varphi^+$ then $\nu(p)=\#\pi_1^{\orb}(X,\varphi)_p$.
\end{prop}

\begin{proof}
We distinguish several cases for $p$ such that
$(\bar{X},p)\cong(\bc^2/A,0)$ where $A$ is a small abelian group (hence cyclic).
Let $\bb_p$ be a small
neighborhood of $p$ (a quotient of a ball $\bb_0$ in $\bc^2$).

Let us suppose that $p\in \mathring{X}_\varphi$.
In this case $\pi_1^{\orb}(\bar{X},\varphi)_p$ is isomorphic to
$\pi_1(K_p)$, where $K_p$ is the link of the singularity $(\bar{X},p)$
which is a lens space with fundamental group~$A$ and the result follows.

Let us assume now that $p$ belongs only to one irreducible component $D_i\subset\DD$ 
where $D_i$ is the image of $Y:=\{y=0\}\subset\bc^2$.  We have a short exact sequence
$$
0\to\pi_1(\bb_0\setminus Y)\to\pi_1(\bb_p\setminus D_i)\to A\to 0.
$$
Both $\pi_1(\bb_0\setminus Y)$ and $\pi_1(\bb_p\setminus D_i)$ are isomorphic to $\bz$
(written with multiplicative notation), which is generated by an element~$t$ which projects
to a generator of $A$. By the definition of the action, the image of a generator of $\pi_1(\bb_0\setminus Y)$
is a meridian~$\mu_i$ of $D_i$ which equals $t^{\iota_p}$. Hence, we obtain $\pi_1^{\orb}(X,\varphi)_p$
from $\pi_1(\bb_p\setminus D_i)$ by killing $x_i^{n_i}=t^{i_p n_i}=t^{\nu(p)}$ and the result follows.

Finally, let us assume that $p$ belongs to two irreducible components $D_i, D_j\subset\DD$ 
where $D_i$ is the image of $Y:=\{y=0\}\subset\bc^2$ and 
$D_j$ is the image of $X:=\{x=0\}$. The covering induces the following short exact sequence:
\begin{equation}
\label{eq-ses}
0\to\pi_1(\bb_0\setminus(X\cup Y))\to\pi_1(\bb_p\setminus (D_i\cup D_j))\to A\to 0.
\end{equation}
Both $\pi_1(\bb_0\setminus Y)$ and $\pi_1(\bb_p\setminus D_i)$ are isomorphic to $\bz^2$
(written with multiplicative notation as above). The group $\pi_1(\bb_0\setminus(X\cup Y))$
is generated by commuting meridians of $X$ and $Y$ whose images are $x_i$ and $x_j$.
We can choose an element~$t\in\pi_1(\bb_p\setminus (D_i\cup D_j))$ which projects
to a generator of $A$. With a suitable choice of~$t$, we have $t^{\iota_p}=x_i x_j^k$
($k$ depends on the specific action and is coprime with $\iota_p$).
Hence~\eqref{eq-ses} induces the following short exact sequence
$$
0\to\langle x,y\mid [x,y]=1,x^{n_j}=x^{n_i}=1\rangle\to\pi_1^{\orb}(X,\varphi)_p\to A\to 0
$$
and the result follows.
\end{proof}

\begin{rem}
Note that if $p$ is an orbifold point of index~$m$, then $\pi_1^{\orb}(\bar{X},\varphi)_p$ is cyclic of order~$m$. 
If $p$ is an ordinary double point of $\DD$ belonging to two components $D_i,D_j$ with $n_i,n_j>1$, then 
$\pi_1^{\orb}(X,\varphi)_p$ is the product of two finite cyclic groups. As a consequence, if $(\bar{X},\varphi)$ 
is a normal crossing orbifold then it is in particular a locally finite orbifold.
\end{rem}

\begin{definition}
A group $G$ is said to be an \emph{orbifold group} if it is isomorphic to 
$\pi_1^{\orb}(\bar{X},\varphi;p_0)$ for some orbifold $(\bar{X},\varphi)$. If one can choose
$(\bar{X},\varphi)$ to be such that $n_i>0$, $\forall i$, then we say that $G$ is a \emph{compact orbifold group}. 
If, moreover $(\bar{X},\varphi)$ is a locally finite (resp. normal crossing orbifold), we say that $G$ is an 
$\LF$ (resp. $\NC$) \emph{compact orbifold group}.
\end{definition}

\begin{rem}
Note that an orbifold group as defined below is also the fundamental group of an orbifold
$(\bar{X},\varphi)$ where $\bar{X}$ is smooth. 
\end{rem}

\begin{rem}\label{rem-blowup1}
We do not define the more general concept of \emph{$\LF$ or $\NC$ orbifold groups} since they coincide immediately
with the concept of \emph{orbifold group} by the following fact. If we blow up a point in $p\in\DD$, we obtain 
a new surface $\bar{Y}$ and a new divisor $\hat{\DD}$ with $r+1$ irreducible components (the strict transforms 
of the components $D_i$, with the same notation, and the exceptional component $D_{r+1}$). We can define a map 
$\hat{\varphi}$ such that $\hat{\varphi}(D_i)=n_i$, $1\leq i\leq r$, and $\hat{\varphi}(D_{r+1})=0$ and the 
orbifold fundamental group does not change. An iterated application of this procedure will give us a normal 
crossing divisor. 
\end{rem}

\begin{prop}\label{prop-7arr}
Let us consider in $\bp^2$ the arrangement of lines $\LL$ given by the equation $x y z (x^2 - z^2) (y^2-z^2)$
and consider the orbifold structure $\varphi_\LL$ given by assigning~$2$ to each line in $\LL$. Let 
$G:=\pi_1^{\orb}(\bp^2,\varphi_\LL)$. The meridians in $G$ of the exceptional components of the 
blowing-ups of the quadruple points of $\LL$ are of infinite order.
\end{prop}

\begin{proof}
It is easy to see that 
$$
G:=
\langle 
x_i,y_j,\gamma_z : x_i^2=y_j^2=\gamma_z^2=[x_i,y_j]=1, \gamma_z=(XY)^{-1} 
\rangle_{i,j=1,2,3},
$$ 
where $X=x_1x_2x_3$, $Y=y_1y_2y_3$ and $x_i$ (resp. $y_j$) $i,j=1,2,3$ are meridians around the vertical
(resp. horizontal) lines and $\gamma_z$ is a meridian around the line at infinity $\{z=0\}$. 
Denote by $\gamma_{E_x}$ (resp. $\gamma_{E_y}$) the meridian in $G$ 
around the exceptional divisor $E_x$ (resp. $E_y$) after blowing up the point $[0:1:0]$ (resp. $[1:0:0]$).
Note that $\gamma_{E_x}=\gamma_zX=Y^{-1}$, $\gamma_{E_y}=\gamma_zY=X^{-1}$. 
By symmetry, it is enough to 
show that $X$ has infinite order in $G$ or equivalently $c:=X^2\in G'$ has infinite order.
Using Reidemeister-Schreier method it is easily seen that
\begin{equation}\label{eq-pres-g'}
 G'=
\langle
a_i,b_j,c\mid\ [a_i,b_j]=1,[a_1,a_2]=[b_1,b_2]=c^4,\ c\text{ central}
\rangle_{i,j=1,2}.
\end{equation}
It is straightforward that $c$ has infinite order.
\end{proof}

Using the same ideas as in Proposition~\ref{prop-orbicurve} we obtain the following result.

\begin{prop}\label{prop-orb-qp}
Any orbifold group is a quasi-projective group.
\end{prop}

\begin{proof}
As above, consider $(\bar{X},\varphi)$ an orbifold for which $\DD=D_1\cup \dots \cup D_r$.
For each divisor $D_j\in S_\varphi^{>1}$ let $\rho_j$ be the $n_j$-fold blow-up on $D_j$ and denote by 
$\rho:\bar{Y}\to \bar{X}$ the composition of all of them. Let us denote by $E_{k,j}$, $1\leq k\leq n_j$, 
$1\leq j\leq r$ the $k$-th exceptional component of~$\rho_j$. Define
$Y:=\bar{Y}\setminus\bigcup_{D_j\in S_\varphi^{>1}}\left( D_j\cup \bigcup_{k=1}^{n_j-1} E_{k,j}\right)$,
where $D_j$ here denotes the strict transform of $D_j$ by $\rho$ and similarly with $E_{k,j}$. 
Note that $\bar{Y}$ is the result of a finite process of blow-ups of a projective variety $\bar{X}$, 
hence $Y$ is quasi-projective variety. Moreover, using Remark~\ref{rem-nfold} it is straightforward to 
check that $Y$ satisfies the required property~$\pi_1(Y)\cong \pi_1^{\orb}(\bar{X},\varphi)$.
\end{proof}

In light of Remark~\ref{rem-fm} and Proposition~\ref{prop-orb-qp}, the following question arises:

\begin{question}
Is any compact orbifold group (or $\NC$-compact orbifold group) a projective group?
\end{question}

A negative answer to the first part is provided by the ideas given
in Example~\ref{ex-nc} and Proposition~\ref{prop-7arr}. This seems to suggest that $\NC$-compact
orbifolds are a reasonable class of orbifolds to work with for our purposes.

\begin{prop}
\label{prop-nonkaeler}
Compact orbifold groups are not necessarily $\NC$-compact orbifold groups, and thus not projective groups.
\end{prop}

\begin{proof}
Consider the compact orbifold group $G$ presented in Proposition~\ref{prop-7arr}. 
The subgroup $G'$ is of finite index. We can describe $G'$ a central extension of $\bz^4$ by $\bz$
as it is deduced from the presentation~\eqref{eq-pres-g'}. If $G$ is an $\NC$-compact orbifold groups,
it is also the case for $G'$. As the only finite-order element of $G'$ is the identity, $G$
should be projective and hence K{\"a}hler. 
One can check that such a group is not 1-formal since it has non-vanishing 
triple Massey products of 1-dim classes. 
Since K\"ahler spaces are formal, and formal spaces have $1$-formal 
fundamental groups (see~\cite{DGMS}), we obtain a contradiction.
\end{proof}

\begin{rem}\label{rem-saturation}
From another point of view, Proposition~\ref{prop-7arr} implies that the local fundamental
group $\pi_1^{\orb}(\bp^2,\varphi_\LL)_{[0:1:0]}$ is infinite and thus the orbifold $(\bp^2,\varphi_\LL)$
has no uniformization in the sense of~\cite[Theorem 2.4]{Uludag-orbifolds}. 
\end{rem}

\begin{rem}
Note that Propositions~\ref{prop-orb-qp} and~\ref{prop-nonkaeler} partially answer questions
posed by Simpson~\cite[\S8]{simpson:11}.
\end{rem}

\section{Saturated orbifolds}\label{sec-saturation}

Since we are mainly interested in orbifold groups it is sometimes useful
to replace in $(\bar{X},\varphi)$ the function $\varphi$ by another
function $\tilde{\varphi}$ where $\tilde{\varphi}(D_i)$ is defined
by the actual order of $\mu_i$ in $\pi_1^{\orb}(\bar{X},\varphi;p_0)$;
we may perform this operation only when $n_i>0$ in order to have
$X_\varphi^+=X_{\tilde{\varphi}}^+$. 
This notion is somehow related with \cite[Condition~(1.3.3)]{namba-branched}.

\begin{definition}\label{def-saturation}
Given an orbifold $(\bar{X},\varphi)$ (for a fixed $\DD$), we say that $\varphi$ is a 
\emph{saturated orbifold structure} if for any meridian $\mu_i$ of $D_i$ (with $n_i>0$), the order of  
$\mu_i$ in $\pi_1^{\orb}(\bar{X},\varphi;p_0)$ is exactly~$n_i$.
\end{definition}

There is a natural way to saturate an orbifold. Unless otherwise stated
we will consider only saturated orbifolds in the sequel. Sometimes an extra saturation
can be performed; even if $n_i=0$, it may happen that $\mu_i$ is of finite order in $\pi_1^{\orb}(\bar{X},\varphi;p_0)$.
Note that in that case if we define $\tilde{\varphi}(n_i)$ to be this order, 
then $X_\varphi^+\subsetneqq X_{\tilde{\varphi}}^+$.

We are going to study different kinds of saturation and their relationship 
with the concept of $\NC$-orbifolds. Let $(\bar{X},\varphi)$ be an orbifold;
if $\DD$ is not a normal crossing divisor there is a sequence $\pi:\bar{Y}\to\bar{X}$
of blowing-ups (which is an isomorphism outside $\mathring{X}$) such that
$\pi^{-1}(\DD)$ becomes a normal crossing divisor. An orbifold structure~$\psi$
can be endowed to $\bar{Y}$ as in Remark~\ref{rem-blowup1}, i.e. $\psi$ vanishes
on any exceptional component of $\pi$. This procedure does not change the orbifold fundamental
group but in general $X_\varphi^+$ and $Y_\psi^+$ are not isomorphic; in particular,
when $(\bar{X},\varphi)$ is not $\NC$, 
$(\bar{Y},\psi)$ is not a compact orbifold even if 
$(\bar{X},\varphi)$ is.

We are going to consider now a more general class of saturations where $X_\varphi^+$ may change
without modifying $\pi_1^{\orb}(\bar{X},\varphi)$.

\begin{definition}\label{def-local-lf}
Let  $(\bar{X},\varphi)$ be an orbifold and let $p\in\bigcup S_\varphi^+$.
Let $\pi:\bar{Y}\to\bar{X}$ be the blowing-up of $p$ and keep the notation
of Remark~\ref{rem-blowup1}.
We say that $p$ is an \emph{$\LF$-point at first level} if the order of
the meridian $\mu_{r+1}$ is finite in $\pi_1(\bar{X},\varphi)_p$.
\end{definition}

Let $\pi:\bar{Y}\to\bar{X}$ be the blowing-up of an $\LF$-point at first level;
let $\hat{\DD}:=\pi^{-1}(\DD)$; with the notation of Remark~\ref{rem-blowup1},
we consider a saturation $\psi$ such that $\psi(D_i):=n_i$, $1\leq i\leq r$
and $\psi(D_{r+1})$ is the order of the meridian~$\mu_{r+1}$ in $\pi_1^{\orb}(\bar{X},\varphi)_p$.

\begin{definition}\label{def-local-lf2}
A point $p$ is an \emph{$\LF$-point}  if all of its infinitely near points are $\LF$-points 
at first level (in particular, if an orbifold is locally finite at a point~$p$ then this point 
is an $\LF$-point).
\end{definition}

\begin{example}\label{ex-lf}
By the very construction~$\pi_1^{\orb}(\bar{X},\varphi)\cong\pi_1^{\orb}(\bar{Y},\psi)$.
Hence if $p$ is an $\LF$-point we can obtain a sequence of blow-ups such that the
divisor $\DD$ becomes a normal crossing divisor \emph{over}~$p$ and such that all the
exceptional divisors have non-zero orbifold indices.
\end{example}

\begin{example}
\label{ex-nc}
If $p\in X_\varphi^+$ is an ordinary double point $p\in D_i\cap D_j$ of $\DD$ then
$\psi(D_{r+1})=\lcm (n_i, n_j)$; if $p\in D_i$ is a smooth point of $\DD$ then $\psi(D_{r+1})=n_i$.
\end{example}

Examples~\ref{ex-nc} and~\ref{ex-lf} show that $\LF$ and $\NC$ compact 
orbifold fundamental groups are the same class of groups.

\begin{example}\label{rem-saturation1}
Locally finiteness may not happen for more complicated singular points.
As a simple example if $p\in\DD$ is an ordinary triple singular point
with orbifold indices for each branch $u,v,w\in\bn$ such that
$\frac{1}{u}+\frac{1}{v}+\frac{1}{w}\leq 1$, then $p$ is not an $\LF$-point.
In the same way if $p$ is an ordinary cusp and the orbifold index is $\geq 6$
then $p$ is not an $\LF$-point.
\end{example}

We set the global version of Definitions~\ref{def-local-lf} and~\ref{def-local-lf2}.

\begin{definition}\label{def-global-lf}
Let  $(\bar{X},\varphi)$ be an orbifold and let $p\in\bigcup S_\varphi^+$.
Let $\pi:\bar{Y}\to\bar{X}$ be the blowing-up of $p$ (keeping again the notation
of Remark~\ref{rem-blowup1}).
We say that $p$ is a \emph{finite-type point at first level} if the order of
the meridian $\mu_{r+1}$ is finite in $\pi_1(\bar{X},\varphi)$.
A point $p$ is a finite-type point if
all of its infinitely near points are finite-type points at first level.
\end{definition}

\begin{rem}\label{rem-saturation2}
Let us consider $\bar{X}=\bp^2$, $\DD$ the union of three lines through a point~$p$
with orbifold indices $u,v,w\in\bn$ such that
$\frac{1}{u}+\frac{1}{v}+\frac{1}{w}\leq 1$. It is clear that~$p$ is not
an $\LF$-point at first level and it is easy to see that it is a finite-type point, since the meridian
around the exceptional component is in fact trivial.
The quadruple points of the arrangement in Proposition~\ref{prop-7arr} are not of finite type. 
Hence the classes of \emph{compact orbifold} and \emph{$\NC$-compact orbifold} groups do not coincide.
\end{rem}

Let us start from a saturated orbifold structure. Hence, if all the points of $\DD$ (it is enough to check 
it for singular points of $\DD$ worse than nodal points)
are $\LF$-points (or finite-type point) we can replace $(\bar{X},\varphi)$ by an $\NC$-orbifold structure
in a surface after successive blowing-ups without changing the fundamental group. 
In the first case we call this structure \emph{locally saturated}; in the second
case it is called \emph{globally saturated}.
Moreover, this can be done respecting the compactness. 

We finish this section with a new saturation procedure which modifies~$\pi_1^{\orb}(\bar{X},\varphi)$.
An interesting object of study associated with $\pi_1^{\orb}(\bar{X},\varphi)$ is the set of its 
characteristic varieties, see~\S\ref{sec-charvar}, which is a stratification of the space of characters in $\bc^*$. 
Since $H_1^{\orb}(\bar{X},\varphi;\bz)_p$ is generated by the meridians
of the components of $\DD$ passing through~$p$, we can associate to~$D_{r+1}$
the order of $\mu_{r+1}$ in $H_1^{\orb}(\bar{X},\varphi;\bz)_p$ (or in $H_1^{\orb}(\bar{X},\varphi;\bz)$).
The orbifold structure is called \emph{locally homologically saturated} or \emph{globally homologically saturated}.

\begin{example}
If we consider an ordinary triple point where all the components have index~$2$, the local homological saturation
is given by assigning~$2$ to the exceptional component. It is easily seen that the local saturation
assigns index~$4$.
\end{example}

\section{Orbifolds and characteristic varieties}\label{sec-charvar}

The relationship between orbifolds and characteristic varieties (or similar invariants)
appear implicitly in the works of Beauville~\cite{Be} and Arapura~\cite{ara:97}
and explicitly in the works of Campana, e.g.~\cite{Campana-special}, Simpson-Corlette~\cite{cs:08},
Delzant~\cite{Delzant} and ourselves~\cite{ACM-prep}, among others. 
Except in Campana's work, the relationship comes from the following fact: given a smooth variety
(projective, quasi-projective or K{\"a}hler) the positive-dimensional components of the 
characteristic varieties can be obtained as pull-back by mappings whose targets are orbifolds.
Campana's work focuses on the study of characteristic varieties of
compact K{\"a}hler orbifolds (more precisely, $\NC$-projective orbifolds in the language of~\S\ref{sec-groups}).
In this section we will study the characteristic varieties of
quasi-projective orbifolds. For a detailed exposition of the concept
of characteristic varieties (or Green-Lazarsfeld invariant), the reader
can check any of the above references. Some definitions will also be given in~\S\ref{sec-sakuma}.

Before we state the aforementioned results in the context and language of orbifolds we need to recall the
concept of orbifold morphism, which as it occurs in the classical case, allows one to define a morphism of
fundamental groups.

\begin{definition}\label{def-map-orb}
Let $(\bar{X},\varphi), (\bar{Y},\psi)$ be orbifolds with divisors $\DD:=\bigcup_{y=1}^r D_j\subset\bar{X}$,
$n_j:=\varphi(D_j)$, 
$\EE:=\bigcup_{k=1}^s E_j\subset\bar{Y}$, $m_k:=\psi(E_k)$. A dominant holomorphic map $\Phi:X_\varphi^+\to Y_\psi^+$
defines an \emph{orbifold map} 
$\Phi^{\orb}:(\bar{X},\varphi)\to(\bar{Y},\psi)$
if for each $k\in\{1,\dots,s\}$, the divisor $\Phi^*(E_k)$ can be written as $\sum_{j=1}^r h_{j,k} D_j+m_k H_k$
where $m_k$ divides $n_j h_{j k}$ and $H_k$ is a divisor in $X_\varphi^+$.
\end{definition}

\begin{prop}[\cite{cko,ACM-multiple-fibers}]
Let $\Phi^{\orb}:(\bar{X},\varphi)\to(\bar{Y},\psi)$ be an orbifold map. This map
induces (in a functorial way) a morphism $\Phi^{\orb}_*:\pi_1^{\orb}(\bar{X},\varphi)\to\pi_1^{\orb}(\bar{Y},\psi)$.
Moreover, if $(\bar{Y},\psi)$ is an orbicurve and the generic fiber of  $\Phi^{\orb}$ is irreducible
then $\Phi^{\orb}_*$ is surjective.
\end{prop}

There are two main examples of orbifold morphisms: either the target is an orbicurve or
the orbifolds have the same dimension. The last case ({\'etale} or branched covers) is specially
interesting when all the fibers are finite. 

Let us compare the following results. We use the language of \S\ref{sec-groups} if needed.

\begin{thm}[{\cite[Theorem~1]{ACM-prep}}]\label{thmprin}
Let $X$ be a smooth quasi-projective variety and let $\mathcal{V}_k(X)$ be the $k$-th characteristic 
variety of $X$. Let $V$ be an irreducible component of~$\mathcal{V}_k(X)$. 
Then one of the two following statements holds:
\begin{enumerate}
\enet{\rm(\arabic{enumi})}
\item\label{thmprin-orb} 
There exists an orbicurve $(\bar{C},\psi)$, an orbifold morphism $\Phi^{\orb}:X\to (\bar{C},\psi)$ and 
an irreducible component $W$ of $\mathcal{V}_k^{\orb}(\bar{C},\psi)$ 
such that $V=(\Phi^{\orb})^*(W)$.
\item\label{thmprin-tors} $V$ is an isolated torsion point not of type~\ref{thmprin-orb}.
\end{enumerate}
\end{thm}

\begin{thm}[{\cite[Th{\'e}or{\`e}me~3.1]{Campana-special}}]
Let $(\bar{X},\varphi)$ be an $\NC$-compact K{\"a}hler orbifold surface.
Let $V$ be an irreducible component of $\mathcal{V}_k^{\orb}(\bar{X},\varphi)$. 
Then, one of the following statements holds:
\begin{enumerate}
\item $V$ is an isolated torsion point.
\item There exists a compact hyperbolic orbicurve $(\bar{C},\psi)$, where the genus of $\bar{C}$
is at least~$1$, an orbifold map $\Phi^{\orb}:(\bar{X},\varphi)\to(\bar{C},\psi)$
and an irreducible component~$W$ of $\mathcal{V}_k^{\orb}(\bar{C},\psi)$
such that $V=(\Phi^{\orb})^*(W)$.
\end{enumerate}
\end{thm}

The goal of this section is to state and prove a combination of the above theorems.

\begin{thm}\label{thm-acm-camp}
Let $(\bar{X},\varphi)$ be an $\NC$-quasi-projective orbifold surface.
Let $V$ be an irreducible component of $\mathcal{V}_k^{\orb}(\bar{X},\varphi)$. 
Then, one of the following statements holds:
\begin{enumerate}
\enet{\rm(\arabic{enumi})}
\item\label{thmprin-new-orb} There exists an orbicurve $(\bar{C},\psi)$, an orbifold map 
$\Phi^{\orb}:(\bar{X},\varphi)\to(\bar{C},\psi)$
and an irreducible component~$W$ of $\mathcal{V}_k^{\orb}(\bar{C},\psi)$
such that $V=(\Phi^{\orb})^*(W)$.
\item $V$ is an isolated torsion point.
\end{enumerate}
\end{thm}

\begin{proof}
Let $(\bar{X},\varphi)$ be an $\NC$-quasi-projective orbifold surface. Let $\DD$ be the hypersurface
defining the orbifold structure where we assume that $\DD=\bigcup_{j=1}^{r+s} D_j$, where
$n_j\geq 2$ if $1\leq j\leq r$ and $n_{r+k}=0$ if $1\leq k\leq s$. We may assume the orbifold
structure is saturated.

We proceed as in the proof of Proposition~\ref{prop-orb-qp}. Let $\pi:\bar{Y}\to\bar{X}$ the composition
of the $\sum_{j=1}^r n_j$ blow-ups indicated in that proof. We denote by $D_i$ the strict transforms
of $D_i$ and by $E_{k,j}$, $1\leq k\leq n_j$, $1\leq j\leq r$, the exceptional components of~$\pi$.
Let $Y:=\bar{Y}\setminus\bigcup_{j=1}^r \left(D_j\cup\bigcup_{k=1}^{n_j-1} E_{k,j}\right)$.
Recall that $\pi_1^{\orb}(\bar{X},\varphi)\cong\pi_1(Y)$.

We can apply Theorem~\ref{thmprin} to $Y$. Let us consider a component $V$ of $\mathcal{V}_k(Y)$
of type~\ref{thmprin-orb} and consider the orbifold map given in the statement. Let us write
this orbifold map in the language of \S\ref{sec-groups}. We consider in $\bar{Y}$ the hypersurface
$$
\hat{\DD}=\bigcup_{j=1}^r\left(D_j\cup\bigcup_{k=1}^{n_j} E_{k,j}\right)\cup\bigcup_{\ell=1}^s D_{r+\ell},
$$
and the map $\hat{\varphi}$ given by:
$$
\underset{1\leq j\leq r+s}{\hat{\varphi}(D_j)=0},\quad
\underset{1\leq k< n_j, 1\leq j\leq r}{\hat{\varphi}(E_{k,j})}=0,\quad
\underset{1\leq j\leq r}{\hat{\varphi}(E_{n_j,j})=1}.
$$
Since $Y_{\hat{\varphi}}^+=Y$, the map given by Theorem~\ref{thmprin} can be written as
$\hat{\Phi}^{\orb}:(\bar{Y},\hat{\varphi})\to(\bar{C},\psi)$.
Let us consider $\hat{\Phi}:Y\to\bar{C}$ the underlying dominant holomorphic mapping.

Note that $\check{E}_j:=E_{n_j,j}\cap Y$ is isomorphic to $\bc^*$. Let us assume that
$\hat{\Phi}_{\check{E}_j}$ is not constant and hence dominant on $\bar{C}$; in particular,
it determines an orbifold morphism $\hat{\Phi}^{\orb}:(E_{n_j,j},\varphi_j)\to (\bar{C},\psi)$
where $\varphi_j$ is the induced orbifold structure, which is the trivial one. The only
possible choices for $(\bar{C},\psi)$ are either $\bc^*$ (with smooth structure) or $\bc_{2,2}$;
the characteristic varieties of these orbifolds are finite and we are led to a contradiction.

Then, we have proven that $\hat{\Phi}_{\check{E}_j}$ is constant and
denote by $p_j\in\bar{C}$ its image. Let us consider a small neighborhood
$\mathcal{U}_j$ of $\bigcup_{k=1}^{n_j-1} E_{k,j}$; this curve is a linear 
chain of rational smooth curves with self-intersection~$-2$ and the space
$\tilde{\mathcal{U}}_j$ obtained from $\mathcal{U}_j$ by contracting the curves is
isomorphic to the quotient of
a neighborhood $\tilde{\mathcal{U}}_j$ of the origin in $\bc^2$ by the action of 
a cyclic group of order~$k_j$. We may lift $\hat{\Phi}$ to a dominant morphism 
$\hat{\Phi}_j:\tilde{\mathcal{U}}_j\setminus\{0\}\to\bar{C}$; it is easily seen that
if $\hat{\Phi}_j$ cannot be extended to the origin, then $\bar{C}\cong\bp^1$
and the characteristic varieties of $(\bar{C},\psi)$ are finite. Since this is not possible,
$\hat{\Phi}_j$ can be extended and $\hat{\Phi}$ can be extended to $\bigcup_{k=1}^{n_j-1} E_{k,j}$
by sending the curve to $p_j$.

A similar argument allows us to extend $\hat{\Phi}$ to the regular part of $\hat{\DD}$ in $D_j$;
moreover it is also possible to extend it to $D_j\cap E_{n_j,j}$ (with image $p_j$).
Finally we can extend it to the double points $D_i\cap D_j$, $1\leq i<j\leq r$.
Moreover, since this map is constant on  $\bigcup_{k=1}^{n_j} E_{k,j}$,
we can contract these divisors (the exceptional divisors of $\pi$) and we obtain
a holomorphic map $\Phi:X_\varphi^+\to\bar{C}$. 

All we are left to do is to check that $\Phi$ defines the required orbifold morphism.
Before we prove this, note that $\Phi_*$ induces a morphism of orbifold fundamental groups. To see this,
let $\mu_j$ be a meridian around $D_j$; note that $\mu_j^{n_j}$ is a meridian of $E_{n_j,j}$ whose image
by $\Phi_*$ is trivial and hence the map induces a morphism of the orbifold fundamental groups. 

Let us assume that $D_j$ is contained in the preimage of $p_j$ and let
us compute its multiplicity in $\Phi^*(p_j)$, say $a_j$. If we compose $\Phi$ and $\pi$ the multiplicity 
of $E_{k,j}$ in the divisor defined by $p_j$ equals $k a_j$. Let $b_j$ the multiplicity of $p_j$ by $\psi$;
the condition of Definition~~\ref{def-map-orb} for orbifold morphism implies that $n_j a_j$ divides $b_j$
which is exactly the needed
condition for $\Phi$. Hence, the required $\Phi^{\orb}:(\bar{X},\varphi)\to(\bar{C},\psi)$ is constructed.
\end{proof}

\section{Unbranched and branched orbifold covers}\label{sec-covers}

One of the advantages of using orbifold fundamental groups is that we can study standard ramified covers
as unbranched orbifold covers. For technical reasons, we restrict our attention to $\NC$-orbifolds.

\begin{definition}
\label{def-unbranched}
We call an orbifold morphism $\pi:(\bar{Y},\varphi_Y)\to (\bar{X},\varphi_X)$ an \emph{orbifold unbranched covering}
if the fibers of $\pi$ are finite and the following equality holds
$$\nu_{(\bar{Y},\varphi_Y)}(y)\cdot\deg \pi_y=\nu_{(\bar{X},\varphi_X)}(x)$$ 
$\forall x\in X_\varphi^+$, $\forall y\in \pi^{-1}(x)$ 
(see Definition~\ref{def-nu}).
\end{definition}

\begin{rem}
For the shake of simplicity we will often
refer to  \emph{orbifold unbranched covering} as \emph{unbranched covering}.
Note that usual \emph{unbranched covering} are actually \emph{orbifold unbranched covering}.
\end{rem}

The main point in Definition~\ref{def-unbranched} is that \emph{orbifold unbranched coverings} behave for 
orbifold fundamental groups as unbranched coverings behave for fundamental groups. In particular, the monodromy 
action completely determines the orbifold unbranched coverings.

\begin{prop}
\label{prop-covers}
An orbifold unbranched covering induces an injective morphism on orbifold fundamental groups. Moreover,
let $(\bar{X},\varphi_X)$ be an orbifold and let us denote $G:=\pi_1^{\orb}(\bar{X},\varphi_X)$. Let $H\subset G$
be a finite-index subgroup; then there is an orbifold unbranched covering $\pi:(\bar{Y},\varphi_Y)\to (\bar{X},\varphi_X)$
such that $\pi_*(\pi_1^{\orb}(\bar{Y},\varphi_Y))=H$. Moreover, $(\bar{Y},\varphi_Y)$ is essentially unique (i.e.,
both $Y_{\varphi_Y}^+$ and $S_{\varphi_Y}^+$ are unique up to isomorphism).
\end{prop}

As in the standard case, the cover is said to be \emph{regular} or \emph{Galois} if $H\unlhd G$; in that case
the group $G/H$ acts on $Y_{\varphi_Y}^+$ with quotient $X_{\varphi_X}^+$.

\begin{prop}
An orbifold unbranched covering satisfies both the path and homotopy lifting properties and
are determined by the monodromy representation $\rho:\pi_1^{\orb}(X,\varphi_X)\to\Sigma_n$, $n:=\# G/H$.
\end{prop}

A proof of these results can be found basically rewriting~\cite[Theorem~1.3.9]{namba-branched} in the 
language of orbifolds instead of in the language of branched coverings.

\begin{example}
Consider $\bp^1_{2,3,5}$ and the subgroup of 
$G:=\pi_1^{\orb}(\bp^1_{2,3,5})=\langle \mu_2, \mu_3, \mu_5 : \mu_2^2=\mu_3^3=\mu_5^5=(\mu_2\mu_3\mu_5)=1, \rangle$ 
given by the kernel of 
\begin{equation}
\array{crcl}
\rho:& G & \to & \Sigma_5\\
&\mu_2 & \mapsto & (1,5)(2,3) \\
&\mu_3 & \mapsto & (1,4,3) \\
&\mu_5 & \mapsto & (1,2,3,4,5). \\
\endarray
\end{equation}
Note that the preimage of the orbifold point of order 2 has three points. For two of them, the local degree of
the map is 2 (and hence their index is 1) whereas on the remaining point the local degree of $\rho$ is~1 (and hence
it should become a point of index~2). Analogously, around the orbifold point of order 3, the preimage has three
points: two of which will have orbifold index 3 and one with orbifold index 1. Finally, around the orbifold point
of index 5, the preimage is a local uniformization. Hence the local conditions on the orbifold points of the covering
are given to satisfy Definition~\ref{def-unbranched}. A simple Euler characteristic computation shows that $\rho$
induces in fact a (non-regular) unbranched covering from $\bp^1_{2,3,3}$ to $\bp^1_{2,3,5}$ of order~5.
\end{example}

\begin{example}
Consider the following morphism:
\begin{equation}
\array{crcl}
\pi : & \bp^1 & \to & \bp^1\\
& [x:y] & \mapsto & [(x^3-y^3)^2:(x^2+y^2)^3] \\
\endarray
\end{equation}
Generically, fibers have 6 different preimages. The special fibers are at
$[1:0]$ (the roots of $(x^2+y^2)^3$), $[0:1]$ (the roots of $(x^3-y^3)^2$), $[1:1]$ 
(the roots of $y^2x^2(2xy+3x^2+3y^2)$), and $[2:1]$ (the roots of $(x^4-2yx^3-2xy^3+y^4)(y+x)^2$).
Therefore this induces a non-regular unbranched covering from $\bp^1_{6(2),2(3)}$ to $\bp^1_{3(2),3}$
of order~6 (where the subindex $k(m)$ stands for $k$ points of index $m$). 
\end{example}

\begin{definition}
An orbifold unbranched cover $\pi:(\bar{Y},\varphi_Y)\to(\bar{X},\varphi_X)$ is a~\emph{uniformization}
of $(\bar{X},\varphi_X)$ if $\bar{Y}$ does not contain points of orbifold index greater than~$1$.
The uniformization will be \emph{Galois} or \emph{regular} if $\pi$ realizes a quotient of $Y_\varphi^+$
by the action of a finite group (which may not act freely).
\end{definition}

\begin{rem}
As in the standard case, a uniformization (or more generally an unbranched cover)
is Galois if and only if the image~$G_Y$ of~$\pi_1^{\orb}(\bar{Y},\varphi_Y)$ in
$G_X:=\pi_1^{\orb}(\bar{X},\varphi_X)$ is a normal subgroup (the group action is carried by $G_X/G_Y$).
Recall that if $\pi$ is a finite Galois uniformization, then the image of a meridian of a component 
with orbifold index $n_i>1$ by the monodromy action is a product of cycles of the same length~$n_i$ 
(with no fixed points). This is not a characterization of Galois uniformization as 
Example~\ref{ex-virtualunif} shows. This condition is called \emph{virtual regularity} in~\cite{ls:89}.

Also note that saturation (see Definition~\ref{def-saturation}) is trivially a necessary condition for the 
existence of a uniformization.
\end{rem}

\begin{example}\label{ex-virtualunif}
Let us consider an orbicurve $(\bar{C},\varphi)$ where $\bar{C}$ is an elliptic curve and the divisor
contains two points of index~$2$. Recall that
$$
\pi_1^{\orb}(\bar{C},\varphi)=\langle a,b,x,y\mid a^2=b^2=1,a b=[x,y]\rangle
$$
Consider the morphism:
\begin{equation}
\array{crcl}
\rho:& \pi_1^{\orb}(\bar{C},\varphi) & \to & \Sigma_4\\
&a,x & \mapsto & (1,2)(3,4) \\
&b & \mapsto & (1,3)(2,4) \\
&y & \mapsto & (1,2,3).
\endarray
\end{equation}
This morphism defines an unbranched orbifold cover; using Riemann-Hurwitz formula the source
of this cover is a Riemann surface of genus~$3$ (with no point of orbifold index greater than~$1$).
This is an example of a uniformization which is virtually regular, but not regular.
\end{example}

For our purposes, a more global and regular definition of unbranched covering will be enough.

\begin{definition}[\cite{Kato-uniformization,Uludag-orbifolds}]
Let $(\bar{X},\varphi)$ be an orbifold. We say $(\bar{X},\varphi')$ is a \emph{suborbifold} of $(\bar{X},\varphi)$ 
(or equivalently $(\bar{X},\varphi)$ is a \emph{superorbifold} of $(\bar{X},\varphi')$)
if $\varphi'(D_i)|\varphi(D_i)$ (meaning there exists $k\in \bz\setminus \{0\}$ such that 
$\varphi(D_i)=k \varphi'(D_i)$ in particular, if $\varphi(D_i)=0$, then $\varphi'(D_i)=0$). 
\end{definition}

On the other side branched orbifold coverings can also be defined. The definitions will be straightforward
for the orbicurve case. 

\begin{definition}
A Galois covering $\pi:\bar{Y}\to \bar{X}$ between two orbifolds $(\bar{X},\varphi_X)$ and $(\bar{Y},\varphi_Y)$
is a \emph{branched orbifold covering} if there exists a superorbifold structure $(\bar{X},\varphi_s)$ for which
$\pi$ defines an unbranched orbifold covering.
\end{definition}

\section{Sakuma's formul{\ae}}\label{sec-sakuma}

Given an orbifold $(\bar{X},\varphi_X)$, we will define $b_1^{\orb}(\bar{X},\varphi_X)$ as the rank of 
the abelianization of $G_{\varphi_X}:=\pi_1^{\orb}(\bar{X},\varphi_X)$, that is, $\rank (G_{\varphi_X}/G'_{\varphi_X})$. 
After taking a superorbifold, all branched orbifold coverings can be assumed to be unbranched.
Consider $\pi:(\bar{Y},\varphi_Y)\to (\bar{X},\varphi_X)$ an unbranched covering. 

Note that, any unbranched covering $\pi:(\bar{Y},\varphi_Y)\to (\bar{X},\varphi_X)$ produces the action of the 
group of deck transformations $G_\varphi$ over $H_1^{\orb}(\bar{Y},\varphi_Y)$ by conjugation, that is,
consider $\bar g\in G_\varphi$ the class of $g\in G_{\varphi_X}$ and $\bar x\in H_1^{\orb}(\bar{Y},\varphi_Y)$ 
the class of $x\in G_{\varphi_Y}$, then $\bar{g}\cdot \bar{x}=\overline{gxg^{-1}}$. Since 
$\overline{gh ax h^{-1}g^{-1}}=\overline{g[h,ax]ag^{-1} gxg^{-1}}=\overline{gxg^{-1}}$, the action is well defined.
This action endows $H_1^{\orb}(\bar{Y},\varphi_Y)$ with a module structure over the group ring $\bz[G_\varphi]$.
After tensoring by $\bc$, the group $H_1^{\orb}(\bar{Y},\varphi_Y)$ acquires a $\bc[G_\varphi]$-module 
structure. 

Recall the definition of the characteristic variety of a finitely presented group $G$. 
Consider a free resolution of a $\bc[H_1(G)]$-module $M$
$$
\bc[H_1(G)]^m \ \rightmap{\phi}\ \bc[H_1(G)]^n \to M,
$$
then $\Char_k(M):=V(F_k(M))$, where $F_k(M)$ is the $k$-th Fitting ideal (or elementary ideal) of $M$ and 
$V(I)$ denotes the zero set of the ideal $I$. 
Recall that $F_k$ is defined as 0 if $k\leq \max\{0,n-m\}$, 1 if $k>n$. Otherwise $F_k$ is
the set of minors of order $(n-k+1)\times (n-k+1)$ of a presentation matrix $A_{\phi}$,
which is an $n\times m$ matrix with coefficients in $\bc[H_1(G)]$. Note that 
$\Char_{k+1}(M)\subseteq \Char_k(M)$ and $\Char_{n+1}(M)=\emptyset$. For any $\xi\in \bc[H_1(G)]$, it is common 
to define as $\nul(M,\xi)$ (nullity of $\xi$) or $d_\xi(M)$ (depth of $\xi$) as the maximum $k\in \bz$ such 
that $\xi \in \Char_k(M)$.

We will denote by $\Char_k^{\orb}(\bar{X},\varphi_X)$ and $\nul^{\orb}(\bar{X},\varphi_X)$, the invariants 
of the $\bc[H_1(G)]$-module $M$ described above where $G$ is the orbifold fundamental group 
$G:=\pi_1^{\orb}(\bar{X},\varphi_X)$.

Unless otherwise stated, all groups orbifold homology groups $H_1^{\orb}$ will be considered as $\bc[G_\varphi]$-modules.
Sakuma's formul\ae~\cite[Theorem 7.3]{Sakuma-homology} (see also~\cite[Proposition~2.5.6]{Eriko-alexander})
can be combined and extended in the following result.

\begin{thm}
\label{thm-sakuma}
Under the above conditions, if $\pi:(\bar{Y},\varphi_Y)\to (\bar{X},\varphi_X)$ is a 
\begin{equation}
\label{eq-sakuma}
b_1^{\orb}(\bar{Y},\varphi_Y)=
b_1^{\orb}(\bar{X},\varphi_X) + 
\sum_{\xi \in \Hom(G_{\varphi},\bc^*)\setminus \{1\}} \nul^{\orb}(\bar{X},\xi)
\end{equation}
where $G_\varphi:=G_{\varphi_X}/G_{\varphi_Y} $,
$\nul^{\orb}(\bar{X},\xi)$ is the depth of $\xi$ considered as a character in 
$\pi_1^{\orb}(\bar{X},\varphi_X)$.
\end{thm}

\begin{rem}
Note that there is a connection between $b_1^{\orb}$ and $b_1$, namely 
$$
b_1^{\orb}(\bar{X},\varphi_X)=b_1(X^+_{\varphi_X})=b_1(\bar{X}\setminus \varphi_X^{-1}(0)).
$$
\end{rem}

\begin{proof}[Proof of Theorem{\rm~\ref{thm-sakuma}}]
The proof offered in~\cite{Sakuma-homology} also works in this context. We will briefly outline 
the original proof.
\begin{itemize}
 \item[\emph{Step} 1.]
From representation theory one has 
$$H_1^{\orb}(\bar{Y},\varphi_Y)\cong \bigoplus_{\xi\in\Hom(G_\varphi,\bc^*)}
\left[H_1^{\orb}(\bar{Y},\varphi_Y)\right]_{\xi},$$
where 
$$\left[H_1^{\orb}(\bar{Y},\varphi_Y)\right]_{\xi}=
\{x\in H_1^{\orb}(\bar{Y},\varphi_Y) \mid g(x)=\xi(g)\cdot x \ \forall g\in G_\varphi\}.$$
 \item[\emph{Step} 2.]
Using Proposition~\ref{prop-covers} there exists an orbifold $(\bar{X}_\xi,\varphi_\xi)$ such that
$$
\left[H_1^{\orb}(\bar{Y},\varphi_Y)\right]_{\xi}\cong \left[H_1^{\orb}(\bar{X}_\xi,\varphi_\xi)\right]_{\xi}.
$$
The orbifold $(\bar{X}_\xi,\varphi_\xi)$ corresponds to the kernel of $G_{\varphi_X}\to G_\varphi\ \rightmap{\xi}\ \bc^*$,
which is of finite index in $G_{\varphi_X}$.
 \item[\emph{Step} 3.]
For $\xi\neq 1$ one has 
\begin{equation*}
\dim_{\bc} \left[H_1^{\orb}(\bar{X}_\xi,\varphi_\xi)\right]_{\xi}=\nul^{\orb}(\bar{X},\xi).\qedhere
\end{equation*}
\end{itemize}
\end{proof}

\begin{rem}
Note that even if $(\bar{X},\varphi_X)$ is not an $\NC$-orbifold, the notion of
unbranched covers may easily be defined as long as orbifolds with (normal) arbitrary singularities 
are allowed. In that case we may also consider
the orbifold $(\hat{X},\hat{\varphi}_X)$ obtained after a sequence of blow-ups such that
the transform of $\DD$ by this sequence of blow-ups is a normal crossing divisor and 
$\hat{\varphi}_X$ is defined by homological saturation. The pull-back of $\pi$ defines
another orbifold $(\hat{Y},\hat{\varphi}_Y)$. Note that $\hat{Y}=\hat{Y}_\varphi^+$ and it
has only abelian quotient singularities; it is a resolution of $\bar{Y}$ which may have more
complicated singularities. There is a natural surjection 
$\pi_1^{\orb}(\hat{Y},\hat{\varphi}_Y)\twoheadrightarrow\pi_1^{\orb}(\bar{Y},\varphi_Y)$
which is not in general an isomorphism. Nevertheless, generalizing Libgober's arguments in~\cite{li:82},
it can be proved that the first Betti numbers coincide. 
\end{rem}

To illustrate Theorem~\ref{thm-sakuma}, we can compute the genus of the uniformization of 
$\bar X:=\bp^1_{d_1,\dots,d_{n+1}}$ in some cases where for instance the abelianization map 
$\pi:(\bar{X}_{\ab},\varphi_{\ab})\to(\bar{X},\varphi_X)$ 
is a uniformization. According to~\cite[Theorem~1.3.43]{namba-branched}
this is the case whenever $d_i$ divides $\lcm (d_1,\dots,d_{i-1},d_{i+1},\dots,d_{n+1})$.
On the one hand one can directly use the Riemann-Hurwitz formula to obtain
$$
\chi(\bar X_{\ab})=2-2g(\bar X_{\ab})=(1-n)\frac{d_1\cdot\ldots\cdot d_{n+1}}{d}+
\sum_{k=1}^{n+1} \frac{d_1\cdot\ldots\cdot d_{n+1}}{d d_k},
$$ 
where $d:=\lcm (d_1,\dots,d_{n+1})$. This implies
\begin{equation}
\label{eq-genus-dn}
g(\bar X_{\ab})=\frac{d_1\cdot\ldots\cdot d_{n+1}}{2d}
\left[-1+\sum_{k=1}^{n+1} \left(1-\frac{1}{d_k}\right)\right]+1.
\end{equation}

\begin{example}
Consider the case $1\leq d_1\leq\dots\leq d_{n}\leq d_{n+1}=d=\lcm(d_1,\dots,d_{n})$. 
This is a particular case of the result mentioned above and hence the universal abelian covering 
$\pi:\bar X_{\ab}\to \bar X:=\bp^1_{d_1,\dots,d_{n+1}}$ is in fact a uniformization. 
Using Theorem~\ref{thm-sakuma} one can obtain $b_1(\bar X_{\ab})=b_1^{\orb}(\bar X_{\ab})$ 
by counting the characters in the orbifold characteristic variety of $\bar X$. Note that the space of 
characters on $\pi_1^{\orb}(\bar X)$ is a union of $(n+1)$-tuples 
$$
\TT:=\{(\xi_1,\dots,\xi_n,\xi_{n+1}) \mid \xi_j\in\mu_{d_j},
\prod_{k=1}^{n+1} \xi_k=1\}\subset (\bc^*)^{n+1},
$$ 
where $\mu_n\subset\bc^*$ is the subgroup of $n$-th roots of unity.
Since the equation in the definition of $\TT$ can always be solved for $\xi_{n+1}$ one has that 
$\TT\cong \mu_{d_1}\times \dots \times \mu_{d_{n}}$. Denote by $\ell(\xi)$ the length of $\xi\in \TT$, that is, 
the number of non-trivial coordinates of $\xi$. From~\cite[Proposition~3.11]{ACM-prep} one deduces that
$\depth(\xi)=\ell(\xi)-2$. Denote by $\ell'(\xi)$ the length of $\xi$ in the first $n$ coordinates,
that is, its length as an element of $\mu_{d_1}\times \dots \times \mu_{d_{n}}$. Note that $\ell(\xi)=\ell'(\xi)+1$
unless its last coordinate is 1, in which case $\ell(\xi)=\ell'(\xi)$. 
Therefore if we define 
$$b_1'(\bar X_{\ab}):=\sum_{\xi\in \TT} \ell'(\xi).$$
Then $b_1'(\bar X_{\ab})-b_1(\bar X_{\ab})=\frac{D}{d}$, where $D:=d_1\cdot\ldots\cdot d_{n}$ which is the order of the kernel of 
the map 
$\mu_{d_1}\times \dots \times \mu_{d_{n}}\to \mu_d$ given by multiplication.
Hence,
$$
b_1'(\bar X_{\ab})=
\sum_{\emptyset\neq I\subseteq \{1,\dots,n\}} (\# I-1) \prod_{i\in I} (d_i-1) 
$$
and thus,
$$
b_1(\bar X_{\ab})=
\left[ \sum_{\emptyset\neq I\subseteq \{1,\dots,n\}} (\# I-1) \prod_{i\in I} (d_i-1) 
\right]- \frac{D}{d}.
$$
Using~\eqref{eq-genus-dn} this implies
$$\frac{D}{d^2}\left[-1+\sum_k \left(1-\frac{1}{d_k}\right)\right]+2=
\left[ \sum_{\emptyset\neq I\subseteq \{1,\dots,n\}} (\# I-1) \prod_{i\in I} (d_i-1) 
\right]- \frac{D}{d}.
$$
\end{example}

For computational purposes, the depth $\nul^{\orb}(\bar{X},\xi)$ can also be obtained
from $\mathring{X}_\varphi$.

\begin{prop}
\label{prop-restriction}
Under the above conditions,
$$\Char_k^{\orb}(X,\varphi)\setminus \{1\} = \Char_k(\mathring{X}_\varphi)\cap \TT_\varphi \setminus \{1\},$$
where $\TT_\varphi$ is the inclusion of $\TT(X,\varphi)$ into $\TT(\mathring{X}_\varphi)$
given by the surjection $\pi_1(\mathring{X}_\varphi)\twoheadrightarrow \pi_1^{\orb}(X,\varphi)$.
\end{prop}

\begin{proof}
The proof is analogous to the one shown in~\cite[Proposition~2.26]{ACT-survey-zariski} for $k=1$.
\end{proof}

\begin{example}
Consider the space ${\mathcal M}$ of sextics with the following combinatorics: 
\begin{enumerate}
\smallbreak\item $\CC$ is a union of a smooth conic $\CC_2$
and a quartic~$\CC_4$. 
\smallbreak\item $\Sing(\CC_4)=\{P,Q\}$ where $Q$ is a cusp of
type $\mathbb A_4$ and $P$ is a node of type~$\mathbb A_1$.
\smallbreak\item $\CC_2 \cap \CC_4=\{Q,R\}$ where $Q$ is a 
$\mathbb D_7$ on $\CC$ and $R$ is a $\mathbb A_{11}$ on~$\CC$.
\end{enumerate}
The space $\mathcal M = \mathcal M^{(1)}\cup \mathcal M^{(2)}$ is a union of two connected components.
Any such sextics $\CC_6^{(i)}= \CC_2^{(i)} \cup \CC_4^{(i)}$ in $\mathcal M^{(i)}$ can be characterized 
by the fact that the conic $q$ passing through $R$ and $Q$ such that 
$\mult_R(q,\CC_2^{(i)})=\mult_R(q,\CC_4^{(i)})=3$, and
$\mult_Q(q,\CC_2^{(i)})=1$ satisfies $\mult_Q(q,\CC_4^{(i)})=3+i$.
The following example is presented in~\cite{ACC-essential}, we refer to it for details.
Consider the orbifolds $(\bp^2,\varphi_i)$, where $\varphi_i(\CC_4^{(i)})=0$ and $\varphi_i(\CC_2^{(i)})=2$.
Using Proposition~\ref{prop-restriction} and ~\cite[Proposition~3.1]{ACT-survey-zariski} it can be checked 
that 
$$
\Char_1^{\orb}(\bp^2,\varphi_i)\setminus \{1\}=
\begin{cases} \emptyset & \text{ if } i=1\\ \{(1,-1)\} & \text{ if } i=2, \end{cases},\quad
\Char_2^{\orb}(\bp^2,\varphi_i)\setminus \{1\}=\emptyset
$$
and hence, using Sakuma's formula~\ref{thm-sakuma} one has 
$$
b_1^{\orb}(Y_i,\varphi_{Y_i})=
\begin{cases}
0 & \text{ if } i=1\\ 1 & \text{ if } i=2,
\end{cases}
$$
where $(Y_i,\varphi_{Y_i})$ denotes the unramified covering of $(\bp^2,\varphi_i)$, 
since $b_1^{\orb}(\bp^2,\varphi_i)=0$. This provides an alternative way to show that $\CC_6^{(1)}$ and $\CC_6^{(2)}$ 
form a Zariski pair, that is, two curves with the same combinatorics but different embedding in~$\bp^2$. In other
words, we prove that $(\bp^2,\CC_6^{(1)})$ and $(\bp^2,\CC_6^{(2)})$ are not homeomorphic by showing that 
$\pi_1^{\orb}(\bp^2,\varphi_1)$ and $\pi_1^{\orb}(\bp^2,\varphi_2)$ are not isomorphic. Note that any homeomorphism
of $\bp^2$ sending $\CC_6^{(1)}$ to $\CC_6^{(2)}$ should send a meridian around $\CC_2^{(1)}$ to a meridian 
around $\CC_2^{(2)}$ and hence $\pi_1^{\orb}(\bp^2,\varphi_1)$ and $\pi_1^{\orb}(\bp^2,\varphi_2)$ should be isomorphic.
\end{example}

We can readily recover, using Theorem~\ref{thm-sakuma} and Proposition~\ref{prop-restriction}, known
computations of the first Betti number of Hirzebruch congruence covers
associated to line arrangements in $\bp^2$, see~\cite{hir:83,Suciu-enumerative}.

For example, consider the orbifold $X=(\bp^2,\varphi)$ associated to
the $6$ lines Ceva arrangement, where $\varphi$ takes value $n$ for all lines.
Then let $Y$ be the orbifold cover associated to the abelianization
$\pi_1^{\orb}(X)\to(\bz/n\bz)^5$. A straightforward counting argument shows that
$b_1^{\orb}(Y)=5(n-1)(n-2)$.

\bibliographystyle{amsplain}
\providecommand{\bysame}{\leavevmode\hbox to3em{\hrulefill}\thinspace}
\providecommand{\MR}{\relax\ifhmode\unskip\space\fi MR }
\providecommand{\MRhref}[2]{%
  \href{http://www.ams.org/mathscinet-getitem?mr=#1}{#2}
}
\providecommand{\href}[2]{#2}

\end{document}